\documentclass[11pt,a4paper]{amsart}
\usepackage{amsmath,amsfonts,amsthm,amsopn,color,amssymb,enumitem}
\usepackage[a4paper]{geometry}
\usepackage{palatino}
\usepackage{graphicx}
\usepackage[colorlinks=true]{hyperref}
\hypersetup{urlcolor=blue, citecolor=red, linkcolor=blue}

\usepackage{cite}
\usepackage{relsize}
\usepackage{esint}
\usepackage{verbatim}
\usepackage{mathrsfs}
\usepackage{xcolor}
\usepackage{tikz}
\newcommand{\W}{ \mathcal W_{r, \lambda}}

\newcommand{\R}{\mathbb{R}}

\newcommand{\RN}{{\mathbb{R}^N}}

\newcommand{\U}{{\mathcal{U}_{\lambda, \xi_1}}}
\newcommand{\Ud}{{\mathcal{U}_{\lambda, \xi}}}
\newcommand{\Uj}{{\mathcal{U}_{\lambda, \xi_j}}}
\newcommand{\Ui}{{\mathcal{U}_{\lambda, \xi_i}}}
\newcommand{\Uh}{{\mathcal{U}_{\lambda, \xi_h}}}

\newcommand{\beq}{\begin{equation}}
\newcommand{\eeq}{\end{equation}}

\usepackage{ifpdf}\ifpdf%
\usepackage[utf8]{inputenc}\UseRawInputEncoding{}\else
\usepackage[UTF8]{ctex}\fi
\usepackage[T1]{fontenc}
\usepackage{lmodern}
\usepackage{amsmath,amsthm,amssymb,amsfonts,geometry,mathrsfs}
\usepackage{hyperref,cite}\hypersetup{colorlinks,allcolors=blue,breaklinks=true}
\usepackage[numbers,square,sort&compress]{natbib}
\DeclareMathAlphabet{\mathbbold}{U}{bbold}{m}n
{\theoremstyle{plain}\newtheorem{theorem}{Theorem}[section]\newtheorem{lemma}[theorem]{Lemma}\newtheorem{proposition}[theorem]{Proposition}}\numberwithin{equation}{section}{\theoremstyle{remark}\newtheorem{remark}{\bf Remark}}
\usepackage{xcolor}

\title[Choquard critical equation]{Infinitely many non-radial solutions to a critical Choquard equation}

\author[S. Caputo]{Sabrina Caputo}
\address{\noindent  Dipartimento di Matematica, Universit\'a degli Studi di Bari Aldo Moro, Italy }
\email{s.caputo28@phd.uniba.it}

\author[G. Vaira]{Giusi Vaira}
\address{\noindent  Dipartimento di Matematica, Universit\'a degli Studi di Bari Aldo Moro, Italy }
\email{giusi.vaira@uniba.it}
\thanks{Work partially supported by 
the MUR-PRIN-P2022YFAJH ``Linear and Nonlinear PDE's: New directions and Applications" and by the INdAM-GNAMPA project ``Fenomeni non lineari: problemi locali e non locali e loro applicazioni",  CUP E5324001950001}

\date{}

\begin{document}

\allowdisplaybreaks{}
\everymath{\displaystyle}

\maketitle

\begin{abstract}
In this paper we study a class of critical Choquard equations with a symmetric potential, i.e. we consider
the equation $$-\Delta u +V(|x|) u =\left(|x|^{-\mu}* |u|^{2^\star_\mu}\right)|u|^{2^\star_\mu-2}u,\quad\mbox{in}\quad\mathbb R^N$$
where $V(|x|)$ is a bounded, nonnegative and symmetric potential in $\RN$ with $N\geq 5$, $0<\mu\leq 4$, $*$ stands for the standard convolution and $2^\star_\mu:=\frac{2N-\mu}{N-2}$ is the upper critical exponent  in the sense of the Hardy - Littlewood - Sobolev inequality. By applying a finite dimensional reduction method we prove that if $r^2V(r)$ has a local maximum point or local minimum point $r_0>0$ with $V(r_0)>0$ then the problem has infinitely many non-radial solutions with arbitrary large energies. 
\end{abstract}

\section{Introduction}
In this paper we consider the nonlinear Choquard equation \begin{equation}\label{C1}-\Delta u +V(x) u =\left(|x|^{-\mu}* |u|^q\right)|u|^{q-2}u,\quad\mbox{in}\quad\mathbb R^N.\end{equation} This equation has attracted the attention and the interest
of a lot of researchers since it arises in various physical models especially in $\mathbb R^3$ with $q=2$ and $\mu=1$. Indeed, in this case, it seems to originate from H. Fr\"ohlich and S. Pekar's model of the polaron, where free electrons in an ionic lattice interact with phonons associated to deformations of the lattice or with the polarization that it creates on the medium (interaction of an electron with its own hole) \cite{F, F1, P}. The Choquard equation was also introduced by Ph. Choquard in 1976 in the modelling of a one-component plasma. The first that considers \eqref{C1} from a mathematical view point was Lieb \cite{L} who proved the existence and uniqueness, up to translations, of the ground state for \eqref{C1} with $N = 3, \mu = 1, q = 2$ and $V$ a positive constant.
The results obtained by Lieb were subsequently improved and extended by Lions \cite{Lions1, Lions2}. In particular, he showed the existence of a sequence of radially symmetric solutions.\\
When $V$ is a constant, Moroz and Van Schaftingen \cite{MVS} proved regularity, positivity and radial symmetry of the ground states for an optimal range of $q$, and described asymptotic decay of solutions at infinity. \\
Ackermann \cite{A}, instead, proposed a new approach to prove the existence of infinitely many geometrically distinct weak solutions in the case of a periodic potential $V$. \\
There are also a lot of works concerning the existence and concentration behavior of semiclassical solutions for the singularly perturbed subcritical Choquard equation (see f.i. \cite{ACTY, AGSY, Cingolani, DGY, MVS1, WW}). Among them, by using a Lyapunov-Schmidt type reduction method, Wei and Winter \cite{WW}constructed families of positive solutions and  Chen \cite{Chen} constructed multiple semiclassical solutions to the Choquard equation with an external potential by Lyapunov-Schmidt reduction argument. Moreover, Luo et al. \cite{Luo} established the uniqueness of positive solutions that concentrate at the non-degenerate critical points of the potential by combining a local type of Pohozaev identities and blow-up techniques.$$$$
We also refer to \cite{MVS2} for a guide regarding the Choquard equation and references therein.\\\\
In what follows we consider the so-called {\it critical Choquard equation}. In order to clarify the meaning of critical we recall the well-known Hardy - Littlewood - Sobolev inequality (see \cite{LL}).
\begin{proposition}\label{HLS}
Let $t, r>1$ and $0<\mu<N$ with $\frac 1 t +\frac \mu N +\frac 1 r =2$. There exists a sharp constant $C=C(t, N, \mu, r)$ such that for any $f\in L^t(\RN)$ and $h\in L^r(\R)$ $$\int_\RN\int_\RN \frac{|f(x)||h(y)| }{|x-y|^\mu}\, dx\, dy \leq C \| f\|_t\|h\|_r.$$ If $t=r=\frac{2N}{2N-\mu}$ then \beq\label{CNmu}C(t, N, \mu, r)=C(N, \mu)=\pi^{\frac \mu 2}\frac{\Gamma\left(\frac N 2-\frac \mu 2\right)}{\Gamma\left( N -\frac \mu 2\right)}\left\{\frac{\Gamma\left(\frac N 2\right)}{\Gamma(N)}\right\}^{-1+\frac \mu N}.\eeq In this case the equality is achieved if and only if $f\equiv C h$ and $$h(x)=A\left(\gamma^2+|x-a|^2\right)^{-\frac{2N-\mu}{2}}$$ for some $A\in\mathbb C$, $0\neq \gamma\in\mathbb R$ and $a\in\mathbb R^N$.
\end{proposition}
Notice that, by the Hardy - Littlewood - Sobolev inequality, the integral $$\int_\RN\int_\RN \frac{|u(x)|^q |u(y)|^q}{|x-y|^\mu}\, dx\, dy$$ is well-defined if $|u|^q\in L^t(\RN)$ for some $t>1$ satisfying $$\frac 2 t +\frac \mu N =2.$$
Thus, for $u\in H^1(\RN)$, by Sobolev embedding theorems, we know that $$2\leq tq\leq 2^\star=\frac{2N}{N-2}$$ namely
$$\frac{2N-\mu}{N}\leq q \leq \frac{2N-\mu}{N-2}.$$
Thus, $\frac{2N-\mu}{N}$ is called the lower critical exponent while $2^*_\mu=\frac{2N-\mu}{N-2}$ is the upper critical exponent in the sense of the Hardy - Littlewood - Sobolev inequality.\\\\
For the upper critical case, by using the moving plane methods in integral developed in \cite{ChenLiOu1, ChenLiOu2}, Lei \cite{Lei}, Du and Yang \cite{DuYang}, Guo et al. \cite{GuoHuPengShuai} classified independently the positive solutions of the critical Hartree equation 
\beq\label{cC} -\Delta u =\left(|x|^{-\mu}* |u|^{2^\star_\mu}\right)|u|^{2^\star_\mu-2}u,\quad \mbox{in}\quad \mathbb R^N\eeq and proved that every positive solution of \eqref{cC} must assume the form
\beq\label{bolla}
\Ud(x)=\alpha_{N,\mu}\frac{\lambda^{\frac{N-2}{2}}}{\left(1+\lambda^2|x-\xi|^2\right)^{\frac{N-2}{2}}},\quad \lambda>0,\,\, x, \xi\in\mathbb R^N\eeq
where $$\alpha_{N, \mu}:=S^{\frac{(N-\mu)(2-N)}{4(N-\mu+2)}}C(N, \mu)^{\frac{2-N}{2(N-\mu+2)}}\left(N(N-2)\right)^{\frac{N-2}{4}}$$
being $S$ the best constant of the Sobolev embedding and $C(N, \mu)$ defined in \eqref{CNmu}.\\\\ The aim of this paper is to establish the existence of infinitely many solutions for the critical Choquard equation with a symmetric potential $V$ of a special form, i.e. $V(x)=V(|x|)$, namely we will consider the following nonlocal equation in $\RN$
\beq\label{CE}
\left\{\begin{aligned} &-\Delta u+V(|x|)u =\left(|x|^{-\mu}* |u(y)|^{2^\star_\mu}\right)|u|^{2^\star_\mu -2}u\quad &\hbox{in}\,\, \RN\\
&u>0 \quad &\hbox{in}\,\, \RN\\
\end{aligned}\right.\eeq
where $V(r)$ is a bounded non-negative function, $N\geq 5$, $0<\mu\leq 4$, $2^\star_\mu=\frac{2N-\mu}{N-2}\geq 2$.\\ 
It follows from the Pohozaev identity that \eqref{CE} has no solution if $(r^2 V(r))'$ has fixed sign and is not identically zero. Therefore, we see that to obtain a solution for \eqref{CE}, it is necessary to assume that $r^2 V (r)$ has either a local maximum, or a local minimum at some $r_0 > 0$. The aim of this paper is to show that this condition also guarantees the existence of infinitely many non-radial solutions.\\
Our main result in this paper can be stated as follows:\\
\begin{theorem}\label{main}
Suppose that $V(|x|)\geq 0$ is bounded and $N\geq 5$. If $r^2 V (r)$ has either an isolated local maximum, or an isolated local minimum at some $r_0 >0$ with $V (r_0) > 0$, then problem \eqref{CE} has infinitely many non-radial solutions.\end{theorem}
\begin{remark}
Theorem \ref{main} is known when $N=6$, $\mu=4$ and the potential $V (|x'|, x'')$ is a bounded and asymmetric nonnegative function in $\mathbb R^+\times \mathbb R^4$ (see \cite{GMVY}). Also in this case we strongly believe that the radial symmetry of $V$ can be relaxed as in \cite{GMVY} or in \cite{PWY}, by using a combination of the reduction method with local Pohozaev identities. Here, in order to avoid the technical difficulties that arise in the presence of the nonlocal term we only consider the symmetric case (as in \cite{CWY}).\\ As far as we know this is the first result concerning the existence of infinitely many solutions by using the reduction scheme in a general way for this kind of nonlocal problems. \end{remark}
\begin{remark}
It is also interesting to consider the case of sign-changing solutions for the Choquard equation since the sign-changing solutions is hard to obtain due to the nonlocal interaction (see \cite{GVS, RVS}).\\ We strongly believe that a finite dimensional argument can be now applied by constructing infinitely many sign-changing solutions for \eqref{cC} or for the analogous problem in a subcritical regime with a symmetric potential where a non-degeneracy for the limit problem is known. This will be done in a forthcoming paper.
\end{remark}
In order to prove Theorem \ref{main}, we will use a widely used technique in singularly perturbed elliptic problems which takes $k$ which is the number of the bubbles in the solutions as a large parameter. This technique was succefully used also in other contexts and for other problems (see f.i. \cite{WeiYan} and \cite{BV}).\\ Since we construct the solutions by using the so-called bubbles (see \eqref{bolla}) then our results provide also an asymptotic behavior of the solutions that we look for (see \eqref{eq:bolla} and \eqref{sol}).\\\\
In the next we will prove Theorem \ref{main} by constructing solutions with large number of bubbles near the sphere $|y|=r_0$. So, in view of the construction of bubbling solutions, we can say that the effect from the critical growth is stronger than the effect from the unboundedness of the domain.\\

We outline the main ides in order to prove Theorem \ref{main}. \\ Let us fix a positive integer $k\geq k_0$ where $k_0$ is large, which is to be determined later.\\
Let $x=(x',x''),\ x'\in \R^2,\ x''\in \R^{N-2}.$ Define
\begin{equation*}
    \begin{aligned}
        H_s:=&\{ u:u\in D^{1,2}(\RN), u\,  \hbox{is even in}\, x_h, \, h=2,...,N,\\
        &u(r\cos\theta, r\sin\theta, x'')=u\left(r\cos\left(\theta+\frac{2\pi}{k}\right),r\sin \left(\theta+\frac{2\pi}{k}\right),x''\right)\}.
    \end{aligned}
\end{equation*}
Let us choose the centers of the bubbles as the $k$ vertices of a regular $k-$ polygon with radius $r$ inside $\RN$ as:
\begin{equation*}
    \xi_j:=\left( r\cos\frac{2(j-1)\pi}{k}, r\sin\frac{2(j-1)\pi}{k},0\right), \quad j=1,...,k,
\end{equation*}
where $0$ is the zero vector in $\R^{N-2}$ and let 
\begin{equation}\label{eq:bolla}
    \W(x):=\sum_{j=1}^k \Uj(x).
\end{equation}
Choose $\delta>0$ small, such that
$$V(|x|)\geq V_0>0,\quad \quad  |x|\in [r_0-2\delta;r_0+2\delta].$$
We always assume that $r\in [r_0-\delta, r_0+\delta]$ and $$\lambda \in [L_0 k^{\frac{N-2}{N-4}}, L_1 k^{\frac{N-2}{N-4}}]$$
for some constants $L_1>L_0>0$.\\ In order to perform the Lyapunov-Schmidt reduction it is very important to understand the kernel of the linearized operator associated to the limit problem \eqref{cC}. Very recently, in \cite{LLTX}, it has been shown that the solutions of \eqref{cC} are non-degenerate provided $2^\star_\mu\geq 2$ (i.e. $\mu\leq 4$), namely if we look for bounded solution of the linearized equation
\begin{equation}\label{lin}
-\Delta Z= \left[2^\star_\mu \left(\frac{1}{|x|^\mu}* \mathcal U^{2^\star_\mu-1}Z\right)\mathcal U^{2^\star_\mu-1}+(2^\star_\mu-1)\left(\frac{1}{|x|^\mu}*\mathcal U^{2^\star_\mu}\right)\mathcal U^{2^\star_\mu-2}Z\right]\quad\mbox{in}\,\,\,\mathbb R^N
\end{equation}
then $Z\in {\rm span}\langle Z_1, \ldots, Z_N, Z_{N+1}\rangle$ where
$$Z_i:=\frac{\partial\mathcal U}{\partial x_i}, i=1, \ldots, N\quad Z_{N+1}=\frac{N-2}{2}\mathcal U + x\cdot\nabla \mathcal U.$$

The paper is organized as follows. In Section \ref{finite}, we perform the finite-dimensional reduction studying the nonlinear projected problem. In Section \ref{ridotto}, we come to the variational reduction procedure and in Section \ref{teo} we prove Theorem \ref{main}. Finally we give a list of some useful estimates in Section \ref{use}.
\section{The finite dimensional reduction}\label{finite}
In what follows we will carry out the finite-dimensional reduction argument in a weighted space in order to obtain a good estimate for the error term. Let
\begin{equation}\label{normastar}
\|u\|_*=\sup_{x\in\mathbb R^N}\left(\sum_{j=1}^k\frac{1}{(1+\lambda|x-\xi_j|)^{\frac{N-2}{2}+\tau}}\right)^{-1}\lambda^{-\frac{N-2}{2}}|u(x)|\end{equation}
and
\begin{equation}\label{normastarstar}
\|h\|_{**}=\sup_{x\in\mathbb R^N}\left(\sum_{j=1}^k\frac{1}{(1+\lambda|x-\xi_j|)^{\frac{N+2}{2}+\tau}}\right)^{-1}\lambda^{-\frac{N+2}{2}}|h(x)|\end{equation} where $\tau=\frac 12$.\\
We let
$$\mathcal B_{**}:=\left\{h\in L^\infty\,:\, \|h\|_{**}<\infty\right\}.$$

We set $$g(u):=\left(\int_\RN \frac{|u(y)|^{2^\star_\mu}}{|x-y|^\mu}\, dy\right)|u|^{2^\star_\mu -2}u.$$ Moreover we also set
$$g'(u)v=2^\star_\mu\left(\int_\RN \frac{|u(y)|^{2^\star_\mu-2}u(y)v(y)}{|x-y|^\mu}\, dy\right)|u|^{2^\star_\mu -2}u+(2^\star_\mu-1)\left(\int_\RN \frac{|u(y)|^{2^\star_\mu}}{|x-y|^\mu}\, dy\right)|u|^{2^\star_\mu-2}v.$$

 Then we can rewrite the problem \eqref{CE} in the following way 
\begin{equation}\label{CE1}
\left\{\begin{aligned} & -\Delta u+V(|x|)u =g(u)\quad \hbox{in}\, \RN\\ &u>0 \quad \hbox{in}\, \RN\\
\end{aligned}\right.\end{equation}

We look for a solution of the problem \eqref{CE1} of the form 
\beq\label{sol}u_k(x)=\W(x)+\phi_k(x)\eeq where $\phi_k$ is a remainder term that satisfies some suitable orthogonality properties.\\ 
We can rewrite the equation in \eqref{CE1} in the following way
$$\mathcal L(\phi)=\mathcal E + \mathcal N(\phi)$$ where
\begin{equation}\label{L}
\mathcal L(\phi)=-\Delta\phi-g'(\W)\phi+V(|x|)\phi
\end{equation}

\begin{equation}\label{E}
\mathcal E=g(\W)-\sum_{j=1}^k g(\Uj)-V(|x|)\W
\end{equation}
and
\begin{equation}\label{N}
\mathcal N(\phi)=g(\W+\phi)-g(\W)-g'(\W)\phi.
\end{equation}
Let us define for any $j=1, \ldots, k$
\beq\label{psi}
\psi_{j, \lambda, 1}:=\frac{\partial\Uj}{\partial r},\quad \psi_{j, \lambda, 2}:=\frac{\partial \Uj}{\partial \lambda}.\eeq
For later uses we remark that 
\beq\label{DERL}
\psi_{j, \lambda, 2}:=\alpha_{N, \mu}\frac{N-2}{2}\lambda^{\frac{N-4}{2}}\frac{1-\lambda^2|x-\xi_j|^2}{(1+\lambda^2|x-\xi_j|^2)^{\frac N 2}}.\eeq
We consider first the following project problem, i.e.
\begin{equation}\label{linpb}
\left\{\begin{aligned} &\mathcal L(\phi_k)=\mathcal E+\mathcal N(\phi_k)+\sum_{\ell=1}^{2}\mathfrak c_{ \ell}\sum_{j=1}^k g'(\Uj)\psi_{j,\lambda,\ell}\quad \mbox{in}\,\, \mathbb R^N\\
&\phi_k\in H_s\\
&\sum_{j=1}^k\int_\RN g'(\Uj)\psi_{j,\lambda,\ell}\phi_k\, dx=0,\quad \ell=1, 2\end{aligned}\right.\end{equation}
for some real numbers $\mathfrak c_{\ell}$.\\
Then, we need to find $r$ and $\lambda$ so that $\mathfrak c_\ell=0$ in order to have a solution of the problem.\\
\subsection{The linear theory}

Here we consider the linear problem 
\begin{equation}\label{linpbh}
\left\{\begin{aligned} &\mathcal L(\phi_k)=h+\sum_{\ell=1}^{2}\mathfrak c_{ \ell}\sum_{j=1}^k g'(\Uj)\psi_{j,\lambda,\ell}\quad \mbox{in}\,\, \mathbb R^N\\
&\phi_k\in H_s\\
&\sum_{j=1}^k\int_\RN g'(\Uj)\psi_{j,\lambda,\ell}\phi_k \, dx=0,\quad \ell=1, 2\end{aligned}\right.\end{equation}
Reasonig as in Lemma 3.1 and Lemma 3.2 of \cite{GMVY} one can show the following invertibility result.
\begin{lemma}\label{inv}
There exists $k_0>0$ and a constant $C>0$ independent of $k$, such that for all $k\geq k_0$ and all $h\in L^\infty(\RN)$, problem \eqref{linpbh} has a unique solution $\phi$. Besides, 
\beq\label{esterr}
\|\phi\|_*\leq C \|h\|_{**}, \quad |\mathfrak c_\ell| \leq \frac{C}{\lambda^{n_\ell}}\|h\|_{**}\eeq
where $n_1:=-1$ and $n_2:=1$.\end{lemma}
In order to use contraction mapping theorem to prove that \eqref{linpb} is uniquely solvable in the set where $\|\phi\|_*$ is small, we need to estimates $\mathcal N(\phi)$ and $\mathcal E$.
\subsection{Estimate of the error term}
Before to estimate the error term we first introduce for every $j=1, \ldots, k$ the sets $$\Omega_j:=\left\{x=(x', x'')\in \mathbb R^2\times \mathbb R^{N-2}\,\,: \left\langle \frac{x'}{|x'|}, \frac{\xi_j'}{|\xi_j'|}\right\rangle \geq \cos\frac \pi k\right\}.$$ 

\begin{lemma}\label{stimaerrore}
    If $N\geq 5,$ then there exists a small $\varepsilon>0$ such that
$$ \|\mathcal E\|_{**}\lesssim \left (\frac{1}{\lambda}\right)^{1+\varepsilon}.
$$
\end{lemma}
\begin{proof}

We have that
$$\mathcal E:=\left[g(\W)-\sum_{j=1}^k g(\Uj)\right]-V(|x|)\W:=J_1-J_2.$$ We start by estimate $J_2$. 
We have that $$\begin{aligned}|J_2|&\lesssim \sum_{j=1}^k \frac{\lambda^{\frac{N-2}{2}}}{(1+\lambda|x-\xi_j|)^{N-2}}\lesssim \lambda^{\frac{N+2}{2}}\sum_{j=1}^k\frac{1}{\lambda^2(1+\lambda|x-\xi_j|)^{N-2}}\\ &\lesssim \left(\frac{1}{\lambda}\right)^{1+\varepsilon} \sum_{j=1}^k\frac{\lambda^{\frac{N+2}{2}}}{(1+\lambda|x-\xi_j|)^{\frac{N+2}{2}+\tau}}\end{aligned}$$ since $N-1-\varepsilon>\frac{N+2}{2}+\tau$ if $\varepsilon>0$ is small. Hence
$$\|J_2\|_{**}\lesssim \left(\frac{1}{\lambda}\right)^{1+\varepsilon}.$$
Now 
$$\begin{aligned} J_1&=\left[|x|^{-\mu}* \left(\sum_{j=1}^k \Uj\right)^{2^\star_\mu}\right]\cdot\left(\sum_{j=1}^k \Uj\right)^{2^\star_\mu-1}-\sum_{j=1}^k \left(|x|^{-\mu}* (\Uj)^{2^\star_\mu}\right)(\Uj)^{2^\star_\mu-1}\\
&=\underbrace{\left[|x|^{-\mu}* \left(\left(\sum_{j=1}^k \Uj\right)^{2^\star_\mu}-\sum_{j=1}^k (\Uj)^{2^\star_\mu}\right)\right]\cdot\left(\sum_{j=1}^k \Uj\right)^{2^\star_\mu-1}}_{(A)}\\
&+\underbrace{\sum_{i\neq j}^k \left(|x|^{-\mu}* (\Uj)^{2^\star_\mu}\right)(\Ui)^{2^\star_\mu-1}}_{(B)}.\end{aligned}$$
We have that for any $x\in\Omega_j$   $$ \left|\left(\sum_{j=1}^k \Uj\right)^{2^\star_\mu}-\sum_{j=1}^k (\Uj)^{2^\star_\mu}\right|\lesssim (\Uj)^{2^\star_\mu-1} \sum_{i\neq j}\Ui +\sum_{i\neq j}(\Ui)^{2^\star_\mu}.$$
Now, by using Lemma \ref{lemmaB1} and by taking $0<\beta\leq \min\{N-\mu+2, N-2\}=N-2$ (we recall that $0<\mu\leq 4$) we get that for any $x\in\Omega_j$ and $i\neq j$
$$\begin{aligned} (\Uj)^{2^\star_\mu-1}\Ui &\lesssim\frac{\lambda^{\frac{N-\mu+2}{2}}}{(1+\lambda|x-\xi_j|)^{N-\mu+2}}\frac{\lambda^{\frac{N-2}{2}}}{(1+\lambda|x-\xi_i|)^{N-2}}\\
&\lesssim\frac{1}{\lambda^\beta|\xi_i-\xi_j|^\beta}\frac{\lambda^{N-\frac\mu 2}}{(1+\lambda|x-\xi_j|)^{2N-\mu-\beta}}\end{aligned}$$ and hence
$$(\Uj)^{2^\star_\mu-1} \sum_{i\neq j}\Ui \lesssim \left(\frac k \lambda\right)^\beta \frac{\lambda^{N-\frac\mu 2}}{(1+\lambda|x-\xi_j|^{2N-\mu-\beta}}.$$ Now by Remark \ref{R2} and Lemma \ref{stima2} with $\theta:=N-\frac\mu 2$ and $\gamma:=N-\beta+(N-\mu)$ 
$$\begin{aligned} |x|^{-\mu}* (\Uj)^{2^\star_\mu-1} \sum_{i\neq j}\Ui&\lesssim \left(\frac k\lambda\right)^\beta \lambda^{\frac \mu 2}\int_{\RN}\frac{1}{|y|^\mu}\frac{1}{(1+|\lambda(x-\xi_j)-y|)^{2N-\mu+\beta}}\\
&\lesssim \left(\frac k\lambda\right)^\beta \frac{\lambda^{\frac \mu 2}}{(1+|\lambda(x-\xi_j)|)^{N-\beta-\eta}}\end{aligned}$$ for $\eta>0$ small.
Moreover by using Remark \ref{R1} we have that
$$ |x|^{-\mu}* \sum_{i\neq j}(\Ui)^{2^\star_\mu}\lesssim \sum_{i\neq j} \frac{\lambda^{\frac \mu 2}}{(1+\lambda|x-\xi_i|)^\mu}.$$ Moreover for any $x\in\Omega_j$ and $i\neq j$ we get for $0<\delta<\frac\mu 2$ and by using Lemma \ref{lemmaB1}
$$\begin{aligned} \sum_{i\neq j}\frac{1}{(1+\lambda|x-\xi_i|)^\mu}&\lesssim \sum_{i\neq j}\frac{1}{(1+\lambda|x-\xi_i|)^{\frac \mu 2}}\frac{1}{(1+\lambda|x-\xi_j|)^{\frac \mu 2}}\\
&\lesssim \sum_{i\neq j}\left(\frac{1}{\lambda^\delta|\xi_i-\xi_j|^\delta}\right)\frac{1}{(1+\lambda|x-\xi_j|)^{\mu-\delta}}\\
&\lesssim \left(\frac k \lambda\right)^\delta \frac{1}{(1+\lambda|x-\xi_j|)^{\mu-\delta}}\end{aligned}$$
Hence $$\begin{aligned} |(A)|&\lesssim \left(\frac k \lambda\right)^\beta \sum_{j=1}^k\frac{\lambda^{\frac \mu 2}}{(1+|\lambda(x-\xi_j)|)^{N-\beta-\eta}}\sum_{h=1}^k (\Uh)^{2^\star_\mu-1}+\left(\frac k \lambda\right)^\delta \sum_{j=1}^k \frac{\lambda^{\frac \mu 2}}{(1+\lambda|x-\xi_i|)^{\mu-\delta}}\sum_{h=1}^k (\Uh)^{2^\star_\mu-1}\\
&\lesssim \left(\frac k \lambda\right)^\beta \lambda^{\frac{N+2}{2}}\sum_{j=1}^k \frac{1}{(1+|\lambda(x-\xi_j)|)^{N-\beta-\eta}}\sum_{h=1}^k\frac{1}{(1+|\lambda(x-\xi_h)|)^{N-\mu+2}}\\
&+\left(\frac k \lambda\right)^\delta \lambda^{\frac{N+2}{2}}\sum_{j=1}^k \frac{1}{(1+\lambda|x-\xi_i|)^{\mu-\delta}}\sum_{h=1}^k\frac{1}{(1+|\lambda(x-\xi_h)|)^{N-\mu+2}}\\
&\lesssim \left(\frac 1 \lambda\right)^{1+\varepsilon}\lambda^{\frac{N+2}{2}}\sum_{j=1}^k \frac{1}{(1+|\lambda(x-\xi_h)|)^{\frac{N+2}{2}+\tau}}.\end{aligned}$$ Indeed, by choosing $\beta:=\frac{N-2}{2}(1+\varepsilon)$ and recalling that $\lambda \sim k^{\frac{N-2}{N-4}}$ we have that $$\left(\frac k \lambda\right)^\beta \sim \left(\frac 1 \lambda\right)^{\frac{2}{N-2}\beta}\sim \left(\frac 1 \lambda\right)^{1+\varepsilon}.$$ Moreover $$N-\beta-\eta+N-\mu+2:=\frac{N+2}{2}+\left(N-\mu +2-\eta-\frac{N-2}{2}\varepsilon\right)>\frac{N+2}{2}+\tau$$ if $\varepsilon>0$ is small enough. Moreover we get also $\delta:=\frac{2}{N-2}(1+\varepsilon)$. Hence again $\left(\frac k \lambda\right)^\delta \sim \left(\frac 1 \lambda\right)^{1+\varepsilon}.$ Moreover $$N+2-\delta:=\frac{N+2}{2}+\tau+ \left(2-\frac{N-2}{2}\varepsilon-\tau\right)>\frac{N+2}{2}+\tau$$ if $\varepsilon>0$ is small enough. 
Hence $$\|(A)\|_{**}\lesssim \left(\frac 1 \lambda\right)^{1+\varepsilon}.$$ For the term (B) we reason as for the second term in (A) and hence we get also that $\|(B)\|_{**}\lesssim \left(\frac 1 \lambda\right)^{1+\varepsilon}.$ Finally $$\|J_1\|_{**}\lesssim \left(\frac 1 \lambda\right)^{1+\varepsilon}.$$
\end{proof}
\subsection{Estimate of the nonlinear term}
\begin{lemma}\label{stimaN}
    If $N\geq 5,$ then
$$ \|\mathcal N(\phi)\|_{**}\lesssim C\|\phi\|_*^{\min\{2, 2^\star_\mu-1\}}.$$\end{lemma}
\begin{proof}
We remark that 
$$\begin{aligned}\mathcal N(\phi)&= \left(|x|^{-\mu}* \left|\W+\phi\right|^{2^\star_\mu}\right)\left|\W+\phi\right|^{2^\star_\mu-2}\left(\W+\phi\right)-\left(|x|^{-\mu}* \W^{2^\star_\mu}\right)\W^{2^\star_\mu-1}\\&-2^\star_\mu\left(|x|^{-\mu}*\W^{2^\star_\mu-1}\phi\right)\W^{2^\star_\mu-1}-(2^\star_\mu-1)\left(|x|^{-\mu}*\W^{2^\star_\mu}\right)\W^{2^\star_\mu-2}\phi\\
&=\left(|x|^{-\mu}*\left(\left|\W+\phi\right|^{2^\star_\mu}-\W^{2^\star_\mu}-2^\star_\mu\W^{2^\star_\mu-1}\phi\right)\right)\left|\W+\phi\right|^{2^\star_\mu-2}\left(\W+\phi\right)\\
&+\left(|x|^{-\mu}* \W^{2^\star_\mu}\right)\left(\left|\W+\phi\right|^{2^\star_\mu-2}\left(\W+\phi\right)-\W^{2^\star_\mu-1}-(2^\star_\mu-1)\W^{2^\star_\mu-2}\phi\right)\\
&+2^\star_\mu \left(|x|^{-\mu}*\W^{2^\star_\mu-1}|\phi|\right)\left(\left|\W+\phi\right|^{2^\star_\mu-2}\left(\W+\phi\right)-\W^{2^\star_\mu-1}\right)\end{aligned}$$
Then by using Lemma \ref{d1} we get $$\begin{aligned}|\mathcal N(\phi)|&\lesssim  \left(|x|^{-\mu}* \W^{2^\star_\mu-2}\phi^2\right)\left(\W^{2^\star_\mu-1}+|\phi|^{2^\star_\mu-1}\right)+ \left(|x|^{-\mu}* |\phi|^{2^\star_\mu}\right)\left(\W^{2^\star_\mu-1}+|\phi|^{2^\star_\mu-1}\right)\\
&+\left(|x|^{-\mu}* \W^{2^\star_\mu}\right)\left\{\begin{aligned}&\left(\W^{2^\star_\mu-3}|\phi|^2 +|\phi|^{2^\star_\mu-1}\right)\quad&\mbox{if}\quad N=5\,\, \mbox{and}\,\,\ 0<\mu<1\\
& |\phi|^{2^\star_\mu-1}\quad&\mbox{if}\quad N=5\,\, \mbox{and}\,\,\ 1\leq \mu\leq 4\,\, \mbox{or}\,\, N\geq 6\end{aligned}\right.\\
&+\left(|x|^{-\mu}*\W^{2^\star_\mu-1}|\phi|\right)\left( |\phi|^{2^\star_\mu-1}+\W^{2^\star_\mu-2}|\phi|\right).\end{aligned}
$$
We have by using Remark \ref{R2} and Lemma \ref{stima2} that for any $\sigma\in(0, 2+2\tau)$
$$\begin{aligned}|x|^{-\mu}* \W^{2^\star_\mu-2}\phi^2&\lesssim \|\phi\|^2_* \lambda^{N-\frac\mu 2}\left(|x|^{-\mu}*\sum_{j=1}^k \frac{1}{\left(1+\lambda|x-\xi_j|\right)^{N+2+2\tau-\mu}}\right)\\
&\lesssim \|\phi\|^2_* \lambda^{\frac\mu 2}\sum_{j=1}^k \frac{1}{\left(1+\lambda|x-\xi_j|\right)^{2+2\tau-\sigma}}\\\end{aligned}$$
Hence
$$\begin{aligned}\left(|x|^{-\mu}* \W^{2^\star_\mu-2}\phi^2\right)\left(\W^{2^\star_\mu-1}+|\phi|^{2^\star_\mu-1}\right)&\lesssim \|\phi\|_*^2\lambda^{\frac{N+2}{2}}\sum_{j=1}^k \frac{1}{\left(1+\lambda|x-\xi_j|\right)^{\frac{N+2}{2}+\tau+(\frac N 2 + 3+\tau-\mu-\sigma)}}\\
&+\|\phi\|^{2^\star_\mu+1}_*\lambda^{\frac{N+2}{2}}\sum_{j=1}^k \frac{1}{\left(1+\lambda|x-\xi_j|\right)^{\frac{N+2}{2}+\tau+( 2+2^\star_\mu\tau-\frac\mu 2-\sigma)}}\\
&\lesssim \|\phi\|^{2}_*\lambda^{\frac{N+2}{2}}\sum_{j=1}^k \frac{1}{\left(1+\lambda|x-\xi_j|\right)^{\frac{N+2}{2}+\tau}}\end{aligned}$$ since $\frac N 2 + 3+\tau-\mu-\sigma>0$ and $2+2^\star_\mu\tau-\frac\mu 2-\sigma>0$ provided $\sigma$ sufficiently small (recall that $0<\mu\leq 4$). \\ Moreover, again by Remark \ref{R2} and by Lemma \ref{stima2} for any $\sigma\in\left(0, \frac\mu 2+2^\star_\mu\tau\right)$
$$\begin{aligned}|x|^{-\mu}*|\phi|^{2^\star_\mu}&\lesssim \|\phi\|_*^{2^\star_\mu}\lambda^{N-\frac\mu 2}\left(|x|^{-\mu}*\sum_{j=1}^k \frac{1}{\left(1+\lambda|x-\xi_j|\right)^{N-\frac \mu 2+2^\star_\mu\tau}}\right)\\
&\lesssim \|\phi\|_*^{2^\star_\mu}\lambda^{\frac \mu 2}\sum_{j=1}^k \frac{1}{\left(1+\lambda|x-\xi_j|\right)^{\frac \mu 2+2^\star_\mu\tau-\sigma}} \end{aligned}$$
and hence 
$$\begin{aligned}\left(|x|^{-\mu}* |\phi|^{2^\star_\mu}\right)\left(\W^{2^\star_\mu-1}+|\phi|^{2^\star_\mu-1}\right)&\lesssim \|\phi\|_*^{2^\star_\mu}\lambda^{\frac{N+2}{2}}\sum_{j=1}^k \frac{1}{\left(1+\lambda|x-\xi_j|\right)^{\frac{N+2}{2}+\tau+(\frac N 2 + 1+(2^\star_\mu-1)\tau-\frac\mu 2-\sigma)}}\\
&\lesssim \|\phi\|^{2^\star_\mu}_*\lambda^{\frac{N+2}{2}}\sum_{j=1}^k \frac{1}{\left(1+\lambda|x-\xi_j|\right)^{\frac{N+2}{2}+\tau}}\end{aligned}$$ provided $\sigma$ sufficiently small.\\
Instead by using Remark \ref{R1}
$$|x|^{-\mu}* \W^{2^\star_\mu}\lesssim \sum_{j=1}^k |x|^{-\mu}* (\Uj)^{2^\star_\mu}\lesssim \sum_{j=1}^k \frac{\lambda^{\frac \mu 2}}{\left(1+\lambda|x-\xi_j|\right)^\mu}.$$
Now if $N=5$ and $0<\mu<1$ then
$$ \begin{aligned}\left(|x|^{-\mu}* \W^{2^\star_\mu}\right)\W^{2^\star_\mu-3}|\phi|^2&\lesssim \|\phi\|^2_*\lambda^{\frac 7 2}\sum_{j=1}^k \frac{1}{\left(1+\lambda|x-\xi_j|\right)^{\frac 7 2 +\tau +\left(\frac 12+\tau\right)}}\\
&\lesssim \|\phi\|^{2}_*\lambda^{\frac{7}{2}}\sum_{j=1}^k \frac{1}{\left(1+\lambda|x-\xi_j|\right)^{\frac{7}{2}+\tau}} .\end{aligned}$$
If, instead, $N=5$ and $1\leq \mu\leq 4$ or if $N\geq 6$ then
$$\begin{aligned} \left(|x|^{-\mu}* \W^{2^\star_\mu}\right)|\phi|^{2^\star_\mu-1}&\lesssim \|\phi\|^{2^\star_\mu-1}_* \lambda^{\frac {N+2}{ 2}}\sum_{j=1}^k \frac{1}{\left(1+\lambda|x-\xi_j|\right)^{\frac{N+2}{2}+\tau+  \left(\frac \mu2+(2^\star_\mu-2)\tau\right)}}\\
&\lesssim \|\phi\|^{2^\star_\mu-1}_*\lambda^{\frac{N+2}{2}}\sum_{j=1}^k \frac{1}{\left(1+\lambda|x-\xi_j|\right)^{\frac{N+2}{2}+\tau}}.\end{aligned}$$
and again by using Lemma \ref{stima2} for any $\sigma\in\left(0, \frac{N-2}{2}+2+\tau\right)$
$$\begin{aligned}|x|^{-\mu}*\W^{2^\star_\mu-1}|\phi|&\lesssim\|\phi\|_*\lambda^{N-\frac\mu 2}\left(|x|^{-\mu}*\sum_{j=1}^k \frac{1}{\left(1+\lambda|x-\xi_j|\right)^{N-\mu +2+\frac{N-2}{2}+\tau}}\right)\\
&\lesssim \|\phi\|_*\lambda^{\frac\mu 2}\sum_{j=1}^k \frac{1}{\left(1+\lambda|x-\xi_j|\right)^{\frac{N-2}{2}+2+\tau-\sigma}}.  \end{aligned}$$
Then
$$\begin{aligned}\left(|x|^{-\mu}*\W^{2^\star_\mu-1}|\phi|\right)\left( |\phi|^{2^\star_\mu-1}+\W^{2^\star_\mu-2}|\phi|\right)&\lesssim \|\phi\|^{2^\star_\mu}_* \lambda^{\frac {N+2}{ 2}}\sum_{j=1}^k \frac{1}{\left(1+\lambda|x-\xi_j|\right)^{\frac{N+2}{2}+\tau+  \left(\frac N 2-\frac \mu2+1+(2^\star_\mu-1)\tau-\sigma\right)}}\\
&+\|\phi\|^{2}_* \lambda^{\frac {N+2}{ 2}}\sum_{j=1}^k \frac{1}{\left(1+\lambda|x-\xi_j|\right)^{\frac{N+2}{2}+\tau+  \left(\frac{N-2}{2}+4-\mu +\tau-\sigma\right)}}\\
&\lesssim \|\phi\|^{2}_*\lambda^{\frac{N+2}{2}}\sum_{j=1}^k \frac{1}{\left(1+\lambda|x-\xi_j|\right)^{\frac{N+2}{2}+\tau}}\end{aligned}$$ provided $\sigma$ sufficiently small.
\end{proof}
We let $$\mathtt E:=\left\{\phi\in H_s\,\,:\,\, \|\phi\|_*\leq C \left(\frac 1 \lambda\right)^{1+\varepsilon},\,\, \varepsilon>0,\,\,\,\sum_{j=1}^k\int_\RN \phi_k \Theta_{j, \ell}\, dx=0,\quad \ell=1, 2\right\}.$$

We are ready to conclude in a standard way the existence of a unique solution of \eqref{linpb}.

\begin{proposition}\label{contrazione}
There is an integer $k_0>0$ such that for any $k\geq k_0$, $\lambda\in [L_0 k^{\frac{N-2}{N-4}}, L_1 k^{\frac{N-2}{N-4}}]$, $r\in [r_0-\delta, r_0+\delta]$, problem \eqref{linpb} has a unique solution $\phi:=\phi_{r, \lambda}\in \mathtt E$ satisfying 
\beq\label{stimaphi}
\|\phi\|_*\lesssim \left(\frac 1 \lambda\right)^{1+\varepsilon}, \quad |\mathfrak c_\ell|\lesssim \left(\frac 1 \lambda\right)^{1+\varepsilon+n_\ell}\eeq\end{proposition}
where $\varepsilon>0$.
\section{The reduced problem}\label{ridotto}
Let us introduce the energy functional whose critical points are solutions to (\ref{CE})
\begin{equation}
    I(u):=\frac{1}{2}\int_\RN (|\nabla u|^2+V(|x|)u^2)\, dx-\frac{1}{2\cdot2^*_\mu} \int_\RN\int_\RN \frac{|u(x)|^{2^*_\mu}|u(y)|^{2^*_\mu}}{|x-y|^\mu}\, dx\, dy
\end{equation}
and the reduced energy 
\begin{equation}
    \mathcal{F}(r,\lambda):=I(\W+\phi)
\end{equation}
where $r=|\xi_1|$ and $\phi$ is the function founded in Proposition \ref{contrazione}. In a standard way it holds the following result.
\begin{lemma}
Let $(\bar r, \bar\lambda)$ be a critical point of $\mathcal F(r, \lambda)$. Then $\mathcal W_{\bar r, \bar\lambda}+\phi$ is a critical point of the enery functional $I(u)$. Also the viceversa holds.\end{lemma}
Now we need to expand the reduced energy.
\begin{lemma}\label{ridotto}
Let $N\geq 5$.   For $k$ large it holds that
    $$\mathcal{F}(r,\lambda)=k\Bigg(A+\frac{B_1V(r)}{\lambda^2}-\sum_{j=1}^k\frac{B_2}{\lambda^{N-2}|\xi_1-\xi_j|^{N-2}}+O\Bigg(\frac{1}{\lambda^{2+\sigma}}\Bigg)\Bigg),$$ where $$A:=\frac 12 \left(1-\frac{1}{2^\star_\mu}\right)C_{N, \mu};\quad B_1:=\frac 12 \int_{\RN}\mathcal U^2;\quad B_2:=\frac 12 D_{N, \mu}$$ where $C_{N, \mu}$ and $D_{N, \mu}$ are positive constants defined respectively in \eqref{CN} and \eqref{DN}. \end{lemma}
\begin{proof}
Since $\phi\in\mathtt E$ solves \eqref{linpb} we get that $$\left\langle I'(\W+\phi), \phi\right\rangle=0,\quad\forall\,\, \phi\in\mathtt E$$ then there is some $t\in (0, 1)$ such that 
$$\begin{aligned}  \mathcal{F}(r,\lambda)&=I(\W)+\frac 12 D^2I(\W+t\phi)(\phi, \phi)\\
&= I(\W) +\frac 12 \int_{\RN}\left(|\nabla\phi|^2+V(|x|)\phi^2\right)-\frac 12 \int_{\RN}g'(\W+t\phi)\phi^2
\end{aligned}$$
Now $\phi$ solves the problem \eqref{linpb} and hence $$\int_{\RN}\left(|\nabla\phi|^2+V(|x|)\phi^2\right)=\int_{\RN}\left(\mathcal E+\mathcal N(\phi)\right)\phi+\int_{\RN}g'(\W)\phi^2$$
Hence
$$\begin{aligned}  \mathcal{F}(r,\lambda)&=I(\W)-\frac 12 \int_{\RN}\left(\mathcal E+\mathcal N(\phi)\right)\phi-\frac 12 \int_{\RN}\left(g'(\W+t\phi)\phi-g'(\W)\phi\right)\phi\\
&=I(\W)+\mathcal O\left(\int_{\RN}|\mathcal E||\phi|+|\mathcal N(\phi)||\phi|\right)+\mathcal O\left(\int_{\RN}|g'(\W+t\phi)\phi-g'(\W)\phi||\phi|\right)
\end{aligned}$$ 
Now by using Lemma \ref{stimaerrore} it follows that $$\begin{aligned}\int_{\RN}|\mathcal E||\phi|&\lesssim \|\mathcal E\|_{**}\|\phi\|_*\lambda^N\int_{\RN}\sum_{j=1}^k\frac{1}{(1+\lambda|x-\xi_j|)^{N+2\tau}}\, dx\\ &\lesssim k \|\mathcal E\|_{**}\|\phi\|_*\int_{\RN}\frac{1}{(1+|y|)^{N+2\tau}}\, dy\\
&\lesssim k\left(\frac 1 \lambda\right)^{2+2\varepsilon}.\end{aligned} $$
Similarly, by using Lemma \ref{stimaN}
$$\int_{\RN}|\mathcal N(\phi)||\phi|\lesssim k\|\mathcal N(\phi)\|_{**}\|\phi\|_*\lesssim k\left(\frac 1 \lambda\right)^{2+2\varepsilon}.$$ Now
$$\begin{aligned}\left(g'(\W+t\phi)-g'(\W)\right)\phi&=2^\star_\mu\underbrace{\left(|x|^{-\mu}*\left(|\W+t\phi|^{2^\star_\mu-1}-\W^{2^\star_\mu-1}\right)\phi\right)|\W+t\phi|^{2^\star_\mu-1}}_{(A)}\\
&+2^\star_\mu\underbrace{\left(|x|^{-\mu}*\W^{2^\star_\mu-1}\phi\right)\left(|\W+t\phi|^{2^\star_\mu-1}-\W^{2^\star_\mu-1}\right)}_{(B)}\\
&+(2^\star_\mu-1)\underbrace{\left(|x|^{-\mu}*\left(|\W+t\phi|^{2^\star_\mu}-\W^{2^\star_\mu}\right)\right)|\W+t\phi|^{2^\star_\mu-2}\phi}_{(C)}\\
&+(2^\star_\mu-1)\underbrace{\left(|x|^{-\mu}*\W^{2^\star_\mu}\right)\left(|\W+t\phi|^{2^\star_\mu-2}-\W^{2^\star_\mu-2}\right)\phi}_{(D)}.\end{aligned}$$
Now reasoning as in Lemma \ref{stimaN} 
$$\begin{aligned}(A)&\lesssim \left(|x|^{-\mu}*\left(\W^{2^\star_\mu-2}|\phi|^2+|\phi|^{2^\star_\mu}\right)\right)\left(\W^{2^\star_\mu-1}+|\phi|^{2^\star_\mu-1}\right)\\
&\lesssim \|\phi\|_*^2\lambda^{\frac{N+2}{2}}\sum_{j=1}^k \frac{1}{\left(1+\lambda|x-\xi_j|\right)^{\frac{N+2}{2}+\tau+(\frac N 2 + 2+\tau-\mu-\sigma)}}\\&+\|\phi\|^{2^\star_\mu+1}_*\lambda^{\frac{N+2}{2}}\sum_{j=1}^k \frac{1}{\left(1+\lambda|x-\xi_j|\right)^{\frac{N+2}{2}+\tau+( 2+2^\star_\mu\tau-\frac\mu 2-\sigma)}}\\
&+\|\phi\|_*^{2^\star_\mu}\lambda^{\frac{N+2}{2}}\sum_{j=1}^k \frac{1}{\left(1+\lambda|x-\xi_j|\right)^{\frac{N+2}{2}+\tau+(\frac N 2 + 1+(2^\star_\mu-1)\tau-\frac\mu 2-\sigma)}}\\
\end{aligned}$$ 
and hence (by a straightforward computation)
$$\begin{aligned}\int_{\RN}\left(|x|^{-\mu}*\left(|\W+t\phi|^{2^\star_\mu-1}-\W^{2^\star_\mu-1}\right)\phi\right)|\W+t\phi|^{2^\star_\mu-1}\phi\, \lesssim k\|\phi\|^2_*\lesssim \frac{k}{\lambda^{2+\sigma}}.\end{aligned}$$
We have also that
$$(B)\lesssim \left(|x|^{-\mu}*\W^{2^\star_\mu-1}|\phi|\right)\left(\W^{2^\star_\mu-2}|\phi|+|\phi|^{2^\star_\mu-1}\right)$$
$$(C)\lesssim \left(|x|^{-\mu}*\left(\W^{2^\star_\mu-1}|\phi|+|\phi|^{2^\star_\mu}\right)\right)\left(\W^{2^\star_\mu-2}|\phi|+|\phi|^{2^\star_\mu-1}\right)$$ and
$$(D)\lesssim \left(|x|^{-\mu}*\W^{2^\star_\mu}\right)\left\{\begin{aligned}&\left(\W^{2^\star_\mu-3}|\phi|+|\phi|^{2^\star_\mu-1}\right)\quad&\mbox{if}\,\, N=5\,\, \mbox{and}\,\, 0<\mu<1\\
&|\phi|^{2^\star_\mu-1}\quad&\mbox{if}\,\, N=5\,\, \mbox{and}\,\, 1\leq\mu\leq 4\quad &\mbox{or}\,\, N\geq 6.\end{aligned}\right.
$$Again, we use the estimate in Lemma \ref{stimaN} and reasoning as for the term (A), it follows that
$$\int_{\RN}|g'(\W+t\phi)\phi-g'(\W)\phi||\phi|=\mathcal O\left(\frac{k}{\lambda^{2+\sigma}}\right).$$

 It remains to evaluate $I(\W)$. We get, by using the symmetry, 
    $$\begin{aligned}\frac 12\int_{\RN}|\nabla \W|^2&=\frac 12\sum_{j=1}\int_{\RN}|\nabla \Uj|^2+\frac 12\sum_{i\neq j}\int_{\RN}\nabla \Ui\nabla \Uj\\
    &=\frac 12\sum_{j=1}^k\int_{\RN}\left(|x|^{-\mu}*(\Uj)^{2^\star_\mu}\right)(\Uj)^{2^\star_\mu}+\frac 12\sum_{i\neq j}^k\int_{\RN}g(\Ui)\Uj\\
   &= \frac 12 kC_{N, \mu}+\frac 12\sum_{i\neq j}^k \int_{\RN}\left(|x|^{-\mu}* (\Ui)^{2^\star_\mu}\right)(\Ui)^{2^\star_\mu-1}\Uj\\
  &=  \frac 12 kC_{N, \mu}+\frac 12\mathfrak d_{N,\mu}\alpha_{N, \mu}^{2^\star_\mu}\sum_{i\neq j}^k \int_{\RN}\frac{\lambda^{\frac{N+2}{2}}}{\left(1+\lambda^2|x-\xi_i|^2\right)^{\frac{N+2}{2}}}\frac{\lambda^{\frac{N-2}{2}}}{\left(1+\lambda^2|x-\xi_j|^2\right)^{\frac{N-2}{2}}}\\
  &=k \left(\frac 12C_{N, \mu}+\frac 12\mathfrak d_{N,\mu}\alpha_{N, \mu}^{2^\star_\mu}\sum_{j=2}^k \int_{\RN}\frac{\lambda^{\frac{N+2}{2}}}{\left(1+\lambda^2|x-\xi_1|^2\right)^{\frac{N+2}{2}}}\frac{\lambda^{\frac{N-2}{2}}}{\left(1+\lambda^2|x-\xi_j|^2\right)^{\frac{N-2}{2}}}\right)\\
  &=k \left(\frac 12 C_{N, \mu}+\frac 12D_{N, \mu}\sum_{j=2}^k \frac{1}{\lambda^{N-2}|\xi_j-\xi_1|^{N-2}}+\mathcal O\left(\sum_{j=2}^k \frac{1}{\lambda^{N-2+\sigma}|\xi_j-\xi_1|^{N-2+\sigma}}\right)\right)
   \end{aligned}$$
    for some $\sigma>0$, where \beq\label{CN}C_{N, \mu}:=\mathfrak d_{N, \mu}\alpha_{N, \mu}^{2^\star_\mu}\int_{\RN}\frac{1}{(1+|y|^2)^{N}}\eeq and \beq\label{DN}D_{N, \mu}:=\mathfrak d_{N, \mu}\alpha_{N, \mu}^{2^\star_\mu}\int_{\RN}\frac{1}{(1+|y|^2)^{\frac{N+2}{2}}}.\eeq Moreover, in a very standard way it follows that
    $$\begin{aligned}\int_{\RN} V(|x|)\W^2&= k\left(\int_{\RN}V(|x|)(\U)^2+\mathcal O\left(\sum_{j=2}^k\int_{\RN} \U \Uj\right)\right)\\
    &= k \left(\frac{V(r)}{\lambda^2}\int_{\RN}\mathcal U^2+\mathcal O\left(\left(\frac 1 \lambda\right)^{2+\sigma}\right)\right).\end{aligned}$$ At the end we analyze the Choquard type term. First, we have that for any $x\in\Omega_1$ 
    \beq\label{e1}
    \W^{2^\star_\mu}:= \left(\U+\sum_{j=2}^k\Uj\right)^{2^\star_\mu}=(\U)^{2^\star_\mu}+{2^\star_\mu}(\U)^{2^\star_\mu-1}\sum_{j=2}^k \Uj +\mathcal O\left((\U)^{\frac{2^\star_\mu}{2}}\sum_{j=2}^k (\Uj)^{\frac{2^\star_\mu}{2}}\right).\eeq Hence
    $$\begin{aligned} \int_{\RN}\left(|x|^{-\mu}* \W^{2^\star_\mu}\right)\W^{2^\star_\mu}&=k\int_{\Omega_1}\left(|x|^{-\mu}* \W^{2^\star_\mu}\right)\W^{2^\star_\mu}\\
    &\hskip-3.0cm =\underbrace{k \int_{\Omega_1}\left(|x|^{-\mu}*(\U)^{2^\star_\mu}\right)(\U)^{2^\star_\mu}}_{(I_1):=C_{N,\mu}}+\underbrace{{2^\star_\mu}k \int_{\Omega_1}\left(|x|^{-\mu}*(\U)^{2^\star_\mu}\right)(\U)^{2^\star_\mu-1}\sum_{j=2}^k \Uj}_{(I_2)} \\
    &\hskip-3.0cm+\underbrace{k\mathcal O\left(\int_{\Omega_1}\left(|x|^{-\mu}*(\U)^{2^\star_\mu}\right) (\U)^{\frac{2^\star_\mu}{2}}\sum_{j=2}^k (\Uj)^{\frac{2^\star_\mu}{2}}\right)}_{(I_3)}\\
    &\hskip-3.0cm+\underbrace{{2^\star_\mu}k \int_{\Omega_1}\left(|x|^{-\mu}*(\U)^{2^\star_\mu-1}\sum_{j=2}^k \Uj \right)(\U)^{2^\star_\mu}}_{(I_4)}\\
    &\hskip-3.0cm+\underbrace{(2^\star_\mu)^2 k \int_{\Omega_1}\left(|x|^{-\mu}*(\U)^{2^\star_\mu-1}\sum_{j=2}^k \Uj \right)(\U)^{2^\star_\mu-1}\sum_{j=2}^k \Uj}_{(I_5)}\\
     &\hskip-3.0cm+\underbrace{\mathcal O\left(k\int_{\Omega_1}\left(|x|^{-\mu}*(\U)^{2^\star_\mu-1}\sum_{j=2}^k \Uj \right)(\U)^{\frac{2^\star_\mu}{2}}\sum_{j=2}^k (\Uj)^{\frac{2^\star_\mu}{2}}\right)}_{(I_6)}\\
      &\hskip-3.0cm+\underbrace{\mathcal O\left(k\int_{\Omega_1}\left(|x|^{-\mu}*(\U)^{\frac{2^\star_\mu}{2}}\sum_{j=2}^k (\Uj)^{\frac{2^\star_\mu}{2}}\right)(\U)^{2^\star_\mu}\right)}_{(I_7)}\\
       &\hskip-3.0cm+\underbrace{\mathcal O\left(k\int_{\Omega_1}\left(|x|^{-\mu}*(\U)^{\frac{2^\star_\mu}{2}}\sum_{j=2}^k (\Uj)^{\frac{2^\star_\mu}{2}}\right)(\U)^{2^\star_\mu-1}\sum_{j=2}^k \Uj\right)}_{(I_8)}\\
        &\hskip-3.0cm+\underbrace{\mathcal O\left(k\int_{\Omega_1}\left(|x|^{-\mu}*(\U)^{\frac{2^\star_\mu}{2}}\sum_{j=2}^k (\Uj)^{\frac{2^\star_\mu}{2}}\right)(\U)^{\frac{2^\star_\mu}{2}}\sum_{j=2}^k (\Uj)^{\frac{2^\star_\mu}{2}}\right)}_{(I_9)}.
     \end{aligned}$$
     Now, reasoning as before,
     $$(I_2)=k\left(2^\star_\mu D_{N, \mu}\sum_{j=2}^k \frac{1}{\lambda^{N-2}|\xi_j-\xi_1|^{N-2}}+\mathcal O\left(\sum_{j=2}^k \frac{1}{\lambda^{N-2+\sigma}|\xi_j-\xi_1|^{N-2+\sigma}}\right)\right)$$
     Moreover by Lemma \ref{lemmaB1} and Remark \ref{R1} $$\begin{aligned}(I_3)&\lesssim k \sum_{j=2}^k\lambda^N\int_{\RN}\frac{1}{(1+\lambda|x-\xi_1|)^{N+\frac\mu 2}}\frac{1}{(1+\lambda|x-\xi_j|)^{N+\frac\mu 2}}\\
     &\lesssim k \sum_{j=2}^k\frac{1}{\lambda^{N-\frac\mu 2}|\xi_1-\xi_j|^{N-\frac\mu 2}}\int_{\RN}\frac{1}{(1+|y|)^{N+\frac\mu 2}}\, dy\\
     &=\mathcal O\left(k \sum_{j=2}^k\frac{1}{\lambda^{N-2+\sigma}|\xi_1-\xi_j|^{N-2+\sigma}}\right)\end{aligned}$$ since $N-\frac\mu 2 >N-2$. 
Again by Lemma \ref{lemmaB1} and Lemma \ref{stima2} we get that 
$$\begin{aligned}\left(|x|^{-\mu}*(\U)^{2^\star_\mu-1}\sum_{j=2}^k \Uj \right)&\lesssim \sum_{j=2}^k\frac{1}{\lambda^{N-2}|\xi_1-\xi_j|^{N-2}}\lambda^{\frac \mu 2}\int_{\RN}\frac{1}{|y|^\mu}\frac{1}{(1+|\lambda(x-\xi_1)-y|)^{N-\mu+2}}\\
&\lesssim \sum_{j=2}^k\frac{1}{\lambda^{N-2}|\xi_1-\xi_j|^{N-2}}\lambda^{\frac \mu 2}\frac{1}{(1+|\lambda(x-\xi_1)|)^{2-\eta}}\\
	\end{aligned}$$ for some $\eta\in (0, 2)$. Then
	
	$$\begin{aligned}(I_4) &={2^\star_\mu}k \int_{\RN}\left(|x|^{-\mu}*(\U)^{2^\star_\mu-1}\sum_{j=2}^k \Uj \right)(\U)^{2^\star_\mu}-{2^\star_\mu}k\int_{\RN\setminus\Omega_1}\left(|x|^{-\mu}*(\U)^{2^\star_\mu-1}\sum_{j=2}^k \Uj \right)(\U)^{2^\star_\mu}\\
	&= {2^\star_\mu}k \int_{\RN}\left(|x|^{-\mu}*(\U)^{2^\star_\mu}\right)(\U)^{2^\star_\mu-1}\sum_{j=2}^k \Uj -{2^\star_\mu}k\int_{\RN\setminus\Omega_1}\left(|x|^{-\mu}*(\U)^{2^\star_\mu-1}\sum_{j=2}^k \Uj \right)(\U)^{2^\star_\mu}\\
		&=k\left({2^\star_\mu}D_{N,\mu}\sum_{j=1}^k \frac{1}{\lambda^{N-2}|\xi_1-\xi_j|^{N-2}}+\mathcal O\left(\sum_{j=2}^k \frac{1}{\lambda^{N-2+\sigma}|\xi_j-\xi_1|^{N-2+\sigma}}\right)+\left(\frac 1 \lambda\right)^{\varepsilon}\sum_{j=2}^k \frac{1}{\lambda^{N-2}|\xi_j-\xi_1|^{N-2}}\right)\end{aligned}$$
	with $\varepsilon>0$, since $$k 2^\star_\mu \int_{\RN}\left(|x|^{-\mu}*(\U)^{2^\star_\mu}\right)(\U)^{2^\star_\mu-1}\sum_{j=2}^k \Uj=(I_2)$$ and $$\begin{aligned}\int_{\RN\setminus\Omega_1}\left(|x|^{-\mu}*(\U)^{2^\star_\mu-1}\sum_{j=2}^k \Uj \right)(\U)^{2^\star_\mu}&\lesssim \sum_{j=1}^k \frac{1}{\lambda^{N-2}|\xi_1-\xi_j|^{N-2}}\int_{|y|>\frac{\lambda|\xi_1-\xi_2|}{2}}\frac{1}{(1+|y|)^{2N-\mu+2-\eta}}\\ &\lesssim \sum_{j=1}^k \frac{1}{\lambda^{N-2}|\xi_1-\xi_j|^{N-2}}\left(\frac 1 \lambda\right)^{N-\mu+2-\eta}\end{aligned}$$ and $N-\mu+2-\eta>0$. Moreover 
	$$\begin{aligned}(I_5)&\lesssim k\sum_{j=2}^k\frac{1}{\lambda^{N-2}|\xi_1-\xi_j|^{N-2}}\lambda^{N}\sum_{j=2}^k\int_{\RN}\frac{1}{(1+\lambda|x-\xi_1|)^{N+4-\mu-\eta}}\frac{1}{(1+\lambda|x-\xi_j|)^{N-2}}\\
     &\lesssim k \sum_{j=2}^k\frac{1}{\lambda^{N-2+\sigma}|\xi_1-\xi_j|^{N-2+\sigma}}\int_{\RN}\frac{1}{(1+|y|)^{N+(N-\mu+2-\eta-\sigma)}}\, dy\\
     &=\mathcal O\left(k \sum_{j=2}^k\frac{1}{\lambda^{N-2+\sigma}|\xi_1-\xi_j|^{N-2+\sigma}}\right)\end{aligned}$$ by choosing $\sigma>0$ sufficiently small so that $N-\mu+2-\eta-\sigma>0$, 
     and 
$$\begin{aligned}(I_6)&\lesssim k \sum_{j=2}^k\frac{1}{\lambda^{N-2}|\xi_1-\xi_j|^{N-2}}\lambda^{N}\sum_{j=2}^k\int_{\RN}\frac{1}{(1+\lambda|x-\xi_1|)^{N +2-\eta-\frac \mu 2}}\frac{1}{(1+\lambda|x-\xi_j|)^{N -\frac\mu 2}}\\
     &\lesssim k \sum_{j=2}^k\frac{1}{\lambda^{N-2+\sigma}|\xi_1-\xi_j|^{N-2+\sigma}}\int_{\RN}\frac{1}{(1+|y|)^{N+(N-\mu+2-\eta-\sigma)}}\, dy\\
          &=\mathcal O\left(k \sum_{j=2}^k\frac{1}{\lambda^{N-2+\sigma}|\xi_1-\xi_j|^{N-2+\sigma}}\right)\end{aligned}$$ again by choosing $\sigma>0$ sufficiently small so that $N-\mu+2-\eta-\sigma>0$.
      Now, by using Lemma \ref{lemmaB1}

     $$\begin{aligned}&|x|^{-\mu}* (\U)^{\frac{2^\star_\mu}{2}}(\Uj)^{\frac{2^\star_\mu}{2}}\lesssim \lambda^{N-\frac \mu 2}\int_{\RN}\frac{1}{|x-y|^\mu}\frac{1}{(1+\lambda|y-\xi_1|)^{N-\frac\mu 2}}\frac{1}{(1+\lambda|y-\xi_j|)^{N-\frac\mu 2}}\\
     &\lesssim \frac{1}{\lambda^{N-2+\sigma}|\xi_1-\xi_j|^{N-2+\sigma}}\lambda^{N-\frac\mu 2}\int_{\RN}\frac{1}{|x-y|^\mu}\frac{1}{(1+\lambda|y-\xi_1|)^{N-\mu+2-\sigma}}\, dy\\
     &+\frac{1}{\lambda^{N-2+\sigma}|\xi_1-\xi_j|^{N-2+\sigma}}\lambda^{N-\frac\mu 2}\int_{\RN}\frac{1}{|x-y|^\mu}\frac{1}{(1+\lambda|y-\xi_j|)^{N-\mu+2-\sigma}}\, dy\\
     &\lesssim \frac{1}{\lambda^{N-2+\sigma}|\xi_1-\xi_j|^{N-2+\sigma}}\lambda^{\frac\mu 2}\int_{\RN}\frac{1}{|y|^\mu}\frac{1}{(1+|\lambda(x-\xi_1)-y|)^{N-\mu+2-\sigma}}\, dy\\
     &+\frac{1}{\lambda^{N-2+\sigma}|\xi_1-\xi_j|^{N-2+\sigma}}\lambda^{\frac\mu 2}\int_{\RN}\frac{1}{|x-y|^\mu}\frac{1}{(1+|\lambda(x-\xi_j)-y|)^{N-\mu+2-\sigma}}\, dy\\
     &\lesssim \frac{1}{\lambda^{N-2+\sigma}|\xi_1-\xi_j|^{N-2+\sigma}}\lambda^{\frac\mu 2}\left(\frac{1}{(1+\lambda|x-\xi_1|)^{2-\sigma-\eta}}+\frac{1}{(1+\lambda|x-\xi_j|)^{2-\sigma-\eta}}\right)
\end{aligned}$$
for $\sigma>0$ sufficiently small and $\eta\in (0, 2-\sigma)$. Then
$$\begin{aligned}(I_7)&\lesssim k\sum_{j=2}^k \frac{1}{\lambda^{N-2+\sigma}|\xi_1-\xi_j|^{N-2+\sigma}}\lambda^{N}\int_{\RN}\frac{1}{(1+\lambda|x-\xi_1|)^{2-\sigma-\eta+2N-\mu }}\, dx\\
&\lesssim k\sum_{j=2}^k \frac{1}{\lambda^{N-2+\sigma}|\xi_1-\xi_j|^{N-2+\sigma}}\int_{\RN}\frac{1}{(1+|y|)^{N+\left(N+2-\sigma-\eta-\mu \right)}}\, dx\\ 
&\lesssim k\sum_{j=2}^k \frac{1}{\lambda^{N-2+\sigma}|\xi_1-\xi_j|^{N-2+\sigma}}\end{aligned}$$

since $\eta$ and $\sigma$  are so that $N+2-\sigma-\eta-\mu=(N-\mu)+(2-\sigma-\eta)>0$.     Again
$$\begin{aligned}(I_8)&\lesssim k\sum_{j=2}^k \frac{1}{\lambda^{N-2+\sigma}|\xi_1-\xi_j|^{N-2+\sigma}}\lambda^{N}\int_{\RN}\frac{1}{(1+\lambda|x-\xi_1|)^{2-\sigma-\eta+2N-\mu }}\, dx\\
&\lesssim k\sum_{j=2}^k \frac{1}{\lambda^{N-2+\sigma}|\xi_1-\xi_j|^{N-2+\sigma}}\end{aligned}$$ and similarly, it follows that $$(I_9)\lesssim k\sum_{j=2}^k \frac{1}{\lambda^{N-2+\sigma}|\xi_1-\xi_j|^{N-2+\sigma}}.$$
    \end{proof}

\begin{lemma}\label{stimader}
Let $N\geq 5$.   For $k$ large it holds that
    $$\frac{\partial\mathcal{F}(r,\lambda)}{\partial\lambda}=k\left(-\frac{2B_1V(r)}{\lambda^3}+\sum_{j=1}^k\frac{B_2(N-2)}{\lambda^{N-1}|\xi_1-\xi_j|^{N-2}}+O\Bigg(\frac{1}{\lambda^{3+\sigma}}\Bigg)\right),$$ where $\sigma>0$ is a fixed constant.
\end{lemma}
\begin{proof}
We have
\beq\label{s3}\begin{aligned}\frac{\partial\mathcal{F}(r,\lambda)}{\partial\lambda}&=\left\langle I'(\W+\phi), \frac{\partial\W}{\partial\lambda}+\frac{\partial\phi}{\partial\lambda}\right\rangle\\
&=\left\langle I'(\W+\phi), \frac{\partial\W}{\partial\lambda}\right\rangle+\sum_{\ell=1}^2\sum_{j=1}^k \mathfrak c_\ell \left\langle g'\left(\Uj\right)\psi_{j, \lambda, \ell}, \frac{\partial\phi}{\partial\lambda}\right\rangle\\\end{aligned}\eeq
Now
$$\begin{aligned}\left\langle I'(\W+\phi), \frac{\partial\W}{\partial\lambda}\right\rangle&=\left\langle I'(\W), \frac{\partial\W}{\partial\lambda}\right\rangle+\int_{\RN}\left(\nabla\phi\nabla \frac{\partial\W}{\partial\lambda}+V(|x|)\phi\frac{\partial\W}{\partial\lambda}\right) \\
&\underbrace{-\int_{\RN}\left(|x|^{-\mu}*|\W+\phi|^{2^\star_\mu}\right)|\W+\phi|^{2^\star_\mu-1}\frac{\partial\W}{\partial\lambda}}_{(A)}\\&\underbrace{+\int_{\RN}\left(|x|^{-\mu}*|\W|^{2^\star_\mu}\right)\W^{2^\star_\mu-1}\frac{\partial\W}{\partial\lambda}}_{(B)}.\end{aligned}$$ Now
$$\int_{\RN}\nabla\phi\nabla \frac{\partial\W}{\partial\lambda}=\sum_{j=1}^k\int_{\RN}\nabla\phi\nabla \frac{\partial\Uj}{\partial\lambda}=\sum_{j=1}^k\int_{\RN}\nabla\phi\nabla \psi_{j,\lambda,2}=0$$ since $\phi\in\mathtt E$. Moreover easily follows that
$$\int_{\RN}V(|x|)\phi\frac{\partial\W}{\partial\lambda}=\mathcal O\left(k \frac{\|\phi\|_*}{\lambda^{2+\varepsilon}}\right)=\mathcal O\left(k \frac{1}{\lambda^{3+\varepsilon}}\right).$$ Moreover
$$\begin{aligned}(A)+(B)&=-\int_{\RN}\left(|x|^{-\mu}*\left(|\W+\phi|^{2^\star_\mu}-\W^{2^\star_\mu}\right)\right)|\W+\phi|^{2^\star_\mu-1}\frac{\partial\W}{\partial\lambda}\\
&-\int_{\RN}\left(|x|^{-\mu}*|\W|^{2^\star_\mu}\right)\left(|\W+\phi|^{2^\star_\mu-1}-\W^{2^\star_\mu-1}\right)\frac{\partial\W}{\partial\lambda}\end{aligned}$$ and then reasoning as in Lemma \ref{ridotto} and by using the fact that $\psi_{j,\lambda, 2}=\mathcal O\left(\frac 1 \lambda \Uj\right)$ we get that  $$(A)+(B)=\mathcal O\left(k \frac{1}{\lambda^{3+\varepsilon}}\right).$$ Now we need to evaluate the other term in \eqref{s3}. To do so we need to use the estimate on $\mathfrak c_\ell$ given in Proposition \ref{contrazione}.\\ Now since $$\sum_{j=1}^k \left\langle g'\left(\Uj\right)\psi_{j, \lambda, \ell}, \frac{\partial\phi}{\partial\lambda}\right\rangle=-\sum_{j=1}^k \left\langle \frac{\partial \left(g'\left(\Uj\right)\psi_{j, \lambda, \ell}\right)}{\partial\lambda}, \phi\right\rangle$$ then we have

$$\begin{aligned}\left|\sum_{j=1}^k \mathfrak c_\ell \left\langle g'\left(\Uj\right)\psi_{j, \lambda, \ell}, \frac{\partial\phi}{\partial\lambda}\right\rangle\right|&\lesssim |\mathfrak c_2| \|\phi\|_* k=\mathcal O\left(\frac{k}{\lambda^{3+\sigma}}\right)\end{aligned}$$ 
At the end $$\begin{aligned}\left\langle I'(\W), \frac{\partial\W}{\partial\lambda}\right\rangle&=\int_{\RN}\nabla \W\nabla\frac{\partial\W}{\partial\lambda}+\int_{\RN}V(|x|)\W \frac{\partial\W}{\partial\lambda}-\int_{\RN}\left(|x|^{-\mu}*\W^{2^\star_\mu}\right)\W^{2^\star_\mu-1}\frac{\partial\W}{\partial\lambda}.\end{aligned}$$
Now, since $$-\Delta\W=\sum_{j=1}^k\left(|x|^{-\mu}*(\Uj)^{2^\star_\mu}\right)(\Uj)^{2^\star_\mu-1}$$ and by using the explicit expression of $\frac{\partial \Uj}{\partial\lambda}$ then
$$\begin{aligned}\int_{\RN}\nabla \W\nabla\frac{\partial\W}{\partial\lambda}&=k\frac{N-2}{2}\alpha_{N, \mu}^{2^\star_\mu}\mathfrak d_{N, \mu}\frac{c_1}{\lambda}-k\frac{N-2}{2}\alpha_{N, \mu}^{2^\star_\mu}\mathfrak d_{N, \mu}\frac{c_2}{\lambda}\sum_{j=2}^k\frac{1}{\lambda^{N-2}|\xi_1-\xi_j|^{N-2}}\\
&+\mathcal O\left(\frac{1}{\lambda}\sum_{j=2}^k\frac{1}{\lambda^{N}|\xi_1-\xi_j|^{N}}\right)\end{aligned}$$
where $$c_1:=\int_{\RN}\frac{1-|y|^2}{(1+|y|^2)^{N+1}},\qquad c_2:=\int_{\RN}\frac{1}{(1+|y|^2)^{\frac{N+2}{2}}}.$$ In a standard way it follows that
$$\int_{\RN}V(|x|)\W \frac{\partial\W}{\partial\lambda}=-k\left(\int_{\mathbb R^N}\mathcal U^2\right)\frac{V(r)}{\lambda^3}+\mathcal O\left(\frac{k}{\lambda^{3+\sigma}}\right).$$
Now it remains to evaluate the Choquard term.\\ We remark that for every $x\in\Omega_1$
$$\W^{2^\star_\mu}=(\U)^{2^\star_\mu}+{2^\star_\mu}(\U)^{2^\star_\mu-1}\sum_{j=2}^k \Uj+\mathcal O\left(\sum_{j=2}^k(\U)^{\frac{2^\star_\mu}{2}}(\Uj)^{\frac{2^\star_\mu}{2}}\right)$$ and 
$$\W^{2^\star_\mu-1}=(\U)^{2^\star_\mu-1}+({2^\star_\mu}-1)(\U)^{2^\star_\mu-2}\sum_{j=2}^k \Uj+\mathcal O\left(\sum_{j=2}^k(\U)^{\frac{2^\star_\mu}{2}-1}(\Uj)^{\frac{2^\star_\mu}{2}}\right)$$ and hence

$$\begin{aligned}&\int_{\RN}\left(|x|^{-\mu}*\W^{2^\star_\mu}
\right)\W^{2^\star_\mu-1} \frac{\partial\W}{\partial\lambda}= k \int_{\Omega_1}\left(|x|^{-\mu}*\W^{2^\star_\mu}\right)\W^{2^\star_\mu} \frac{\partial\W}{\partial\lambda}\\
\end{aligned}$$
Now again
$$\begin{aligned}\int_{\RN}\left(|x|^{-\mu}*(\U)^{2^\star_\mu}
\right)(\U)^{2^\star_\mu-1} \frac{\partial\W}{\partial\lambda}&=k\frac{N-2}{2}\alpha_{N, \mu}^{2^\star_\mu}\mathfrak d_{N, \mu}\frac{c_1}{\lambda}-k\frac{N-2}{2}\alpha_{N, \mu}^{2^\star_\mu}\mathfrak d_{N, \mu}\frac{c_2}{\lambda}\sum_{j=2}^k\frac{1}{\lambda^{N-2}|\xi_1-\xi_j|^{N-2}}\\
&+\mathcal O\left(\frac{1}{\lambda}\sum_{j=2}^k\frac{1}{\lambda^{N}|\xi_1-\xi_j|^{N}}\right)\end{aligned}$$
Now
$$\begin{aligned}(A)&=k(2^\star_\mu-1)\sum_{j=2}^k\int_{\RN}\left(|x|^{-\mu}*(\U)^{2^\star_\mu}\right)(\U)^{2^\star_\mu-2}\Uj\frac{\partial\U}{\partial\lambda}\\
&=k({2^\star_\mu}-1)\sum_{j=2}^k\mathfrak d_{N, \mu}\alpha_{N, \mu}^{2^\star_\mu}\frac{N-2}{2}\int_{\RN}\frac{\lambda^{N-1}(1-\lambda^2|x-\xi_1|^2)}{(1+\lambda^2|x-\xi_1|^2)^{\frac{N+4}{2}}}\frac{1}{(1+\lambda^2|x-\xi_j|^2)^{\frac{N-2}{2}}}\\
&=k({2^\star_\mu}-1)\mathfrak d_{N, \mu}\alpha_{N, \mu}^{2^\star_\mu}\frac{N-2}{2}\sum_{j=2}^k\frac{1}{\lambda^{N-1}|\xi_1-\xi_j|^{N-2}}\int_{\RN}\frac{1-|y|^2}{(1+|y|^2)^{\frac{N+4}{2}}}\end{aligned}$$ while
$$\begin{aligned}(B)&= k2^\star_\mu\sum_{j=2}^k\int_{\RN}\left(|x|^{-\mu}*(\U)^{2^\star_\mu-1}\Uj\right)(\U)^{2^\star_\mu-1}\frac{\partial\U}{\partial\lambda}\\
&=k\sum_{j=2}^k\int_{\RN}\left(\frac{\partial}{\partial\lambda}\left(|x|^{-\mu}*(\U)^{2^\star_\mu}\right)\right)\U\Uj\\
&=k\frac\mu 2\mathfrak d_{N, \mu}\alpha_{N, \mu}^{2^\star_\mu}\sum_{j=2}^k\int_{\RN}\lambda^{N-1}\frac{1-\lambda^{2}|x-\xi_1|^2}{(1+\lambda^2|x-\xi_1|^2)^{\frac{N+4}{2}}}\frac{1}{(1+\lambda^2|x-\xi_j|^2)^{\frac{N-2}{2}}}\\
&=k\frac\mu 2\mathfrak d_{N, \mu}\alpha_{N, \mu}^{2^\star_\mu}\sum_{j=2}^k\frac{1}{\lambda^{N-1}|\xi_1-\xi_j|^{N-2}}\int_{\RN}\frac{1-|y|^2}{(1+|y|^2)^{\frac{N+4}{2}}}.\end{aligned}$$ Now let
$$I_m^\alpha:=\int_0^{+\infty}\frac{\rho^\alpha}{(1+\rho^2)^{m}}\, d\rho,\quad \mbox{for}\,\,\alpha+1<2m$$ then, it is known (see \cite{Almaraz} Lemmas 9.4 and 9.5) that $$I^\alpha_m=\frac{2m}{2m-\alpha-1}I^\alpha_{m+1}\quad\mbox{for}\,\, \alpha+1<2m+2.$$ Then by using this we have that $I^{N-1}_{\frac{N+2}{2}}=\frac{N+2}{2}I^{N-1}_{\frac{N+4}{2}}.$ Hence
$$\begin{aligned}\int_{\RN}\frac{1-|y|^2}{(1+|y|^2)^{\frac{N+4}{2}}}&=\omega_N \int_0^{+\infty}\frac{\rho^{N-1}(1-\rho^2)}{(1+\rho^2)^{\frac{N+4}{2}}}\\ &=-\omega_N \int_0^{+\infty}\frac{\rho^{N-1}}{(1+\rho^2)^{\frac{N+2}{2}}}+2\omega_N\int_0^{+\infty}\frac{\rho^{N-1}}{(1+\rho^2)^{\frac{N+4}{2}}}\\
&=\omega_N\left(-1+\frac{4}{N+2}\right)\int_0^{+\infty}\frac{\rho^{N-1}}{(1+\rho^2)^{\frac{N+2}{2}}}=-\frac{N-2}{N+2}\int_{\RN}\frac{1}{(1+|y|^2)^{\frac{N+2}{2}}}.\end{aligned}$$
By summing the two terms we get
$$\begin{aligned}(A)+(B)&=k\mathfrak d_{N, \mu}\alpha_{N, \mu}^{2^\star_\mu}\sum_{j=2}^k\frac{1}{\lambda^{N-1}|\xi_1-\xi_j|^{N-2}}\left((2^\star_\mu-1)\frac{N-2}{2}+\frac\mu 2\right)\int_{\RN}\frac{1-|y|^2}{(1+|y|^2)^{\frac{N+4}{2}}}\\
&=-k D_{N, \mu}\sum_{j=2}^k\frac{1}{\lambda^{N-1}|\xi_1-\xi_j|^{N-2}}\end{aligned}$$ For the other terms one can reason as in Lemma \ref{ridotto} by using the fact that $\frac{\partial\W}{\partial\lambda}=\mathcal O\left(\frac 1 \lambda \Uj\right)$.
\end{proof}
\section{Proof of the Theorem \ref{main}}\label{teo}
Since
\begin{align*}
  \left|\xi_{j}-\xi_{1}\right|=2\left|\xi_{1}\right| \sin \frac{(j-1) \pi}{k}, \quad j=1, \ldots, k,
\end{align*}
we have
\begin{align*}
\begin{aligned}
& \sum_{j=2}^{k} \frac{1}{\left|\xi_{j}-\xi_{1}\right|^{N-2}}=\frac{1}{\left(2\left|\xi_{1}\right|\right)^{N-2}} \sum_{j=2}^{k} \frac{1}{\left(\sin \frac{(j-1) \pi}{k}\right)^{N-2}} \\
& \quad= \begin{cases}\frac{1}{\left(2\left|\xi_{1}\right|\right)^{N-2}} \sum_{j=2}^{\frac{k}{2}} \frac{1}{\left(\sin \frac{(j-1) \pi}{k}\right)^{N-2}}+\frac{1}{\left(2\left|\xi_{1}\right|\right)^{N-2}}, & \text{if}~ k ~\text{is even}, \\
\frac{1}{\left(2\left|\xi_{1}\right|\right)^{n-2}} \sum_{j=2}^{ \lfloor \frac{k}{2}\rfloor} \frac{1}{\left(\sin \frac{(j-1) \pi}{k}\right)^{N-2}}, & \text {if}~ k ~\text{is odd},\end{cases}
\end{aligned}
\end{align*}
and
\begin{align*}
  0<c' \leq \frac{\sin \frac{(j-1) \pi}{k}}{\frac{(j-1) \pi}{k}} \leq c'', \quad j=2, \ldots,\left[\frac{k}{2}\right],
\end{align*}
where
$c',c''$ are some positive constants.

So, there exists a constant $B_0>0$, such that
\begin{align*}
\sum_{j=2}^{k} \frac{1}{\left|\xi_{j}-\xi_{1}\right|^{N-2}}=\frac{B_0 k^{N-2}}{\left|\xi_{1}\right|^{N-2}}+O\left(\frac{k}{\left|\xi_{1}\right|^{N-2}}\right)=\frac{B_0 k^{N-2}}{r^{N-2}}+O\left(\frac{k}{r^{N-2}}\right).
\end{align*}
since $|\xi_1|=r$.\\
Then by Lemma \ref{ridotto} and Lemma \ref{stimader} it follows that
\beq\label{F}\mathcal{F}(r,\lambda)=k\left(A+\frac{B_1V(r)}{\lambda^2}-\frac{B_0 k^{N-2}}{\lambda^{N-2}r^{N-2}}+O\left(\left(\frac 1 \lambda\right)^{2+\sigma}\right)\right),\eeq
and

\beq\label{derF}\frac{\partial\mathcal{F}(r,\lambda)}{\partial\lambda}=k\left(-\frac{2B_1V(r)}{\lambda^3}+\frac{B_0(N-2) k^{N-2}}{\lambda^{N-1}r^{N-2}}+O\left(\left(\frac 1 \lambda\right)^{3+\sigma}\right)\right).\eeq
For each fixed $r\in[r_0-\delta, r_0+\delta]$ let $\Lambda_0(r)$ be the solution of $$-\frac{2B_1V(r)}{\Lambda^3}+\frac{B_0(N-2) k^{N-2}}{\Lambda^{N-1}r^{N-2}}=0.$$ Then $$\Lambda_0(r):=\left(\frac{B_0(N-2)}{2B_1V(r) r^{N-2}}\right)^{\frac{1}{N-4}}.$$ We have that $\Lambda_0(r)$ is the unique maximum point of the function $$V(r)\frac{B_1}{\Lambda^2}-\frac{B_0}{\Lambda^{N-2}r^{N-2}}.$$

\begin{proof}[Proof of Theorem \ref{main}]
{\bf Case 1:} Let $r_0$ be a maximum point of $r^2V(r)$.\\
Let us then consider the problem $$\max_{(r, \mu)\in D}\mathcal F(r, \lambda)$$ where $$D:=\left\{(r, \lambda)\,\,:\,\, r\in [r_0-\delta, r_0+\delta]\,:\, \lambda=\Lambda k^{\frac{N-2}{N-4}},\,\,\, \Lambda\in \left[\Lambda_0(r)-\delta_1, \Lambda_0(r)+\delta_1\right]\right\}$$ where $\delta_1:=\frac{1}{k^{\frac{N-2}{N-4}\frac 32\theta}}$ with $0<\theta<<\sigma$ ($\sigma$ is the constant in \eqref{F} and \eqref{derF}).
Reasoning as in the proof of Theorem 1.2 of \cite{CWY} we have that for any $(r, \lambda)\in D$ \beq\label{B'}\frac{B_1V(r)}{\lambda^2}-\frac{B_0 k^{N-2}}{\lambda^{N-2}r^{N-2}}=\left(B'\left(r^2V(r))^{\frac{N-2}{2(N-4)}}+\mathcal O\left(\frac{1}{\lambda^{3\theta}}\right)\right)\frac{1}{k^{\frac{2(N-2)}{N-4}}}\right)\eeq with $B'>0$. Moreover if $\theta>0$ is small enough then from \eqref{derF} it follows that $$\frac{\partial\mathcal{F}(r,\lambda)}{\partial\lambda}>0\quad \mbox{if}\,\, \bar\lambda_k =k^{\frac{N-2}{N-4}}\left(\Lambda_0(r)-\delta_1\right)$$
while 
$$\frac{\partial\mathcal{F}(r,\lambda)}{\partial\lambda}<0\quad \mbox{if}\,\, \bar\lambda_k =k^{\frac{N-2}{N-4}}\left(\Lambda_0(r)+\delta_1\right).$$
So $$\bar\lambda_k \neq k^{\frac{N-2}{N-4}}\left(\Lambda_0(r)\pm \delta_1\right).$$
Hence, since $r^2V(r)$ has a maximum at $r_0$ from \eqref{F}, we see that $\bar r_k\neq r_0\pm\delta$ for the maximum point $(r_k, \bar\lambda_k)\in D$. So, $(r_k, \bar\lambda_k)$ is an interior point of $D$ and thus a critical point of $\mathcal{F}(r,\lambda)$.\\\\

{\bf Case 2:} Let $r_0$ be a local minimum point of $r^2V(r)$.\\\\ Again we reason as in \cite{CWY}. Indeed we define $$\bar{\mathcal{F}}(r,\lambda):=-\mathcal{F}(r,\lambda),\quad (r, \lambda)\in D.$$
where $D$ is as before. Let $$\alpha_2:=k (-A+\eta),\qquad \alpha_1:=k\left(-A-B'(r_0^2 V(r_0)^{\frac{N-2}{2(N-4)}}(1-\eta)\frac{1}{k^{\frac{2(N-2)}{N-4}}}\right),$$ where $\eta>0$ is a small constant and $B'>0$ is as in \eqref{B'}. We let also the sublevel $$\bar{\mathcal{F}}^\alpha:=\left\{(r, \lambda)\in D\,\,\;:\,\,\, \bar{\mathcal{F}}(r, \lambda)\leq \alpha\right\}.$$ Consider $$\left\{\begin{aligned} &\frac{d r}{r t}=-D_r\bar{\mathcal{F}},\quad t>0\\
&\frac{d\lambda}{dt}=-D_\lambda \bar{\mathcal{F}},\quad t>0\\
&(r,\lambda)\in \bar{\mathcal{F}}^{\alpha_2}\end{aligned}\right.$$
It is possible to show that the flow $(r(t), \lambda(t))$ does not leave $D$ before it reaches $\bar{\mathcal{F}}^{\alpha_1}$ (see Proposition 3.3 of \cite{CWY}).\\ Now it is possible to show that $\bar{\mathcal{F}}$ and hence $\mathcal F$ has a critical point in $D$.\\
Define $$\Gamma:=\left\{h\,\,:\,\, h(r, \lambda)=(h_1(r, \lambda), h_2(r, \lambda))\in D, \,\, (r, \lambda)\in D\,\, h(r, \lambda)=(r, \lambda),\,\, \mbox{if}\,\,\, |r-r_0|=\delta\right\}.$$ Let $$c:=\inf_{h\in\Gamma}\max_{(r, \lambda)\in D}\bar{\mathcal F}(h(r, \lambda)).$$
Then $c$ is a critical value of $\bar{\mathcal F}$. As in \cite{CWY} one can show that $$\alpha_1<c<\alpha_2;\qquad \sup_{|r-r_0|=\delta}\bar{\mathcal F}(h(r, \lambda))<\alpha_1\,\, \forall\, h\in\Gamma.$$ The thesis follows.
\end{proof}

\section{Useful Estimates}\label{use}
We recall an useful identity.

\begin{remark}\label{R1}
In \cite{DHQWF} it was shown that (see formula (37)) 
\begin{equation}\label{id}
\int_\RN \frac{1}{|x-y|^{2s}}\left(\frac{1}{1+|y|^2}\right)^{N-s}\, dy =\mathcal I(s) \left(\frac{1}{1+|x|^2}\right)^s,\quad 0<s<\frac N 2\end{equation} where
$$\mathcal I(s):=\frac{\pi^{\frac N 2}\Gamma\left(\frac{N-2s}{2}\right)}{\Gamma(N-s)},\quad \Gamma(s)=\int_0^{+\infty}x^{s-1}e^{-x}\, dx,\,\, s>0.$$
Hence it follows that
\beq\label{imp1}
|x|^{-\mu}* \mathcal U_{\lambda, \xi}^{2^\star_\mu}(x)=\alpha_{N, \mu}\mathcal I\left(\frac \mu 2\right)\left(\frac{\lambda}{1+\lambda^2|x-\xi|^2}\right)^{\frac \mu 2}.
\eeq
Indeed, by using \eqref{bolla} and \eqref{id}
$$\begin{aligned}
|x|^{-\mu}* \mathcal U_{\lambda, \xi}^{2^\star_\mu}(x)&=\alpha_{N, \mu}\int_{\RN}\frac{1}{|x-y|^\mu}\frac{\lambda^{N-\frac\mu 2}}{(1+\lambda^2|x-\xi|^2)^{N-\frac \mu 2}}\, dy\\
&=\alpha_{N, \mu}\lambda^{-\frac \mu 2}\int_{\RN}\frac{1}{\left|x-\frac z \lambda-\xi\right|^\mu}\frac{1}{(1+|z|^2)^{N-\frac \mu 2}}\, dz\\
&=\alpha_{N, \mu}\lambda^{\frac \mu 2}\int_{\RN}\frac{1}{\left|\lambda (x-\xi)- z\right|^\mu}\frac{1}{(1+|z|^2)^{N-\frac \mu 2}}\, dz\\
&=\mathfrak d_{N, \mu}\left(\frac{\lambda}{1+\lambda^2|x-\xi|^2}\right)^{\frac \mu 2}.
\end{aligned}$$ where $\mathfrak d_{N, \mu}:=\alpha_{N, \mu}\mathcal I\left(\frac \mu 2\right)$.
\end{remark}
Now we introduce some estimates involving the convolution term.
\begin{lemma}[Lemma B.1 of \cite{WeiYan}]\label{lemmaB1}
For each fixed $i$ and $j$ with $i\neq j$ we let $$f_{i, j}(x)=\frac{1}{(1+|x-\xi_i|)^\alpha}\frac{1}{(1+|x-\xi_j|)^\beta}$$ where $\alpha\geq 1$ and $\beta\geq 1$ are two constants. Then, for any $0<\sigma\leq \min\{\alpha, \beta\}$, there is a constant $C>0$ such that 
$$f_{i, j}(x)\leq \frac{C}{|\xi_i-\xi_j|^\sigma}\left(\frac{1}{(1+|x-\xi_i|)^{\alpha+\beta-\sigma}}+\frac{1}{(1+|x-\xi_j|)^{\alpha+\beta-\sigma}}\right).$$

\end{lemma}
In the following Lemma we generalize some estimates that are known only for some special parameter.

\begin{lemma}\label{stima2}
    There is a constant $C>0$ such that
$$\int_{\RN}\frac{1}{|y|^\mu}\frac{1}{(1+|\lambda(x-\xi)-y|)^{\alpha+\eta}}\, dy\leq C\frac{1}{(1+\lambda|x-\xi|)^{\alpha-N+\mu}}$$ for $\alpha\geq N-\mu$ and $\eta>0.$
\end{lemma}
\begin{proof}
    Let $d:=\frac\lambda 2 |x-\xi|>1$.
Then, if $y\in B_d(0)$ since $$|\lambda(x-\xi)-y|>\lambda|x-\xi|-|y|>d$$ we get
$$\int_{B_d(0)}\frac{1}{|y|^\mu}\frac{1}{(1+|\lambda(x-\xi)-y|)^{\alpha+\eta}}\, dy\leq \frac{C}{(1+d)^{\alpha+\eta}}\int_{B_d(0)}\frac{1}{|y|^\mu}\, dy\leq \frac{C}{(1+d)^{\alpha-N+\mu+\eta}}.$$
Now, if $y\in B_d(\lambda(x-\xi))$ we get $$|y|=|y-\lambda(x-\xi)+\lambda(x-\xi)|>\lambda|x-\xi|-|y-\lambda(x-\xi)|>d$$ we get
$$\int_{B_d(\lambda(x-\xi))}\frac{1}{|y|^\mu}\frac{1}{(1+|\lambda(x-\xi)-y|)^{\alpha+\eta}}\, dy\leq \frac{1}{d^\mu}\int_{B_d(0)}\frac{1}{(1+|y|)^{\alpha+\eta}}\, dy\leq \frac{C}{(1+d)^{\alpha-N+\mu+\eta}}.$$
If $y\in\RN\setminus (B_d(0)\cup B_d(\lambda(x-\xi))$ then $|y|\geq d=\frac\lambda 2 |x-\xi| $ and $|y-\lambda(x-\xi)|\geq d=\frac\lambda 2 |x-\xi|$.\\ Then we get 
$$\frac{1}{|y|^\mu}\frac{1}{(1+|\lambda(x-\xi)-y|)^{\alpha+\eta}}\leq \frac{1}{(1+d)^{\alpha-N+\mu}}\frac{1}{|y|^\mu}\frac{1}{(1+|\lambda(x-\xi)-y|)^{N-\mu+\eta}}.$$
If $|y|\leq 2\lambda|x-\xi|$ then $|\lambda(x-\xi)-y|\geq |y|-\lambda|x-\xi|\geq \frac 12|y|$. Hence
$$\frac{1}{|y|^\mu}\frac{1}{(1+|\lambda(x-\xi)-y|)^{N-\mu+\eta}}\leq \frac{C}{|y|^\mu}\frac{1}{(1+|y|)^{N-\mu+\eta}}.$$ If $|y|\geq 2 \lambda|x-\xi|$ then $|\lambda(x-\xi)-y|\geq |y|-\lambda|x-\xi|\geq \frac 12 |y|$. Again we get
$$\frac{1}{|y|^\mu}\frac{1}{(1+|\lambda(x-\xi)-y|)^{N-\mu+\eta}}\leq \frac{C}{|y|^\mu}\frac{1}{(1+|y|)^{N-\mu+\eta}}.$$ Thus we have
$$\begin{aligned}&\int_{\RN\setminus (B_d(0)\cup B_d(\lambda(x-\xi))}\frac{1}{|y|^\mu}\frac{1}{(1+|\lambda(x-\xi)-y|)^{\alpha+\eta}}\, dy\\
&\leq \frac{1}{(1+d)^{\alpha-N+\mu}}\int_{\RN\setminus (B_d(0)\cup B_d(\lambda(x-\xi))}\frac{C}{|y|^\mu}\frac{1}{(1+|y|)^{N-\mu+\eta}}\, dy\\
&\leq \frac{C}{(1+d)^{\alpha-N+\mu}}.\end{aligned}$$ The thesis follows.
\end{proof}

\begin{remark}\label{R2}
We remark that $$\begin{aligned}|x|^{-\mu}*\frac{\lambda^\theta}{(1+\lambda|x-\xi_j|)^\gamma}&=\int_{\RN}\frac{1}{|x-y|^\mu}\frac{\lambda^\theta}{(1+\lambda|y-\xi_j|)^\gamma}\, dy \\&=\int_{\RN}\frac{\lambda^{-N+\theta}}{\left|x-\frac z \lambda-\xi_j\right|^\mu}\frac{1}{(1+|z|)^\gamma}\, dz\\
&=\lambda^{-N+\theta+\mu}\int_{\RN}\frac{1}{|y|^\mu}\frac{1}{(1+|\lambda(x-\xi_j)-y|)^\gamma}\, dy\end{aligned}$$

\end{remark}

\begin{lemma}\label{d1}
Let $q$ be a positive real number. Then there exists a constant $C>0$ such that for any $a, b\in\mathbb R$ 
$$\left|\left|a+b\right|^q-|a|^q\right|\leq C \left\{\begin{aligned} & \min\left\{|b|^q, |a|^{q-1}|b|\right\}\quad &\mbox{if}\quad 0<q<1\\
&|a|^{q-1}|b|+|b|^q\quad &\mbox{if}\quad 1\leq q\leq 2\end{aligned}\right.$$Moreover
$$\left|\left|a+b\right|^q-|a|^q-q|a|^{q-2}ab\right|\leq C \left\{\begin{aligned} & \left(|a|^{q-2}|b|^2+|b|^q\right)\quad &\mbox{if}\quad q>2\\
&|b|^q\quad &\mbox{if}\quad q\leq 2\end{aligned}\right.$$
\end{lemma}

\end{document}